%% file: draft.tex
\documentclass[preprint,11pt]{elsarticle}

\usepackage{siunitx}
\usepackage{booktabs}
\usepackage{amsthm}
\usepackage{subcaption}

\usepackage{amsmath,amssymb,bm}
\usepackage{accents}% for under-bar

\biboptions{sort}
\usepackage[usenames,dvipsnames,svgnames,table]{xcolor}

\usepackage{color}

\usepackage{tikz}
\usetikzlibrary{plotmarks}
\usetikzlibrary{positioning}
\usetikzlibrary{decorations.pathreplacing}
\usetikzlibrary{math}

\usepackage{graphicx}
\usepackage{calc,fancyvrb}
\usepackage[margin=1.in]{geometry}

\usepackage{algorithm,algpseudocode}
\usepackage{algorithmicx}
\algrenewcommand\alglinenumber[1]{\footnotesize #1:} % Algorithm line number font size

\usepackage{pgfplots}
\usepackage{graphicx}
\pgfplotsset{compat=1.16}
\usepackage[bookmarks=true,colorlinks=true,linkcolor=blue]{hyperref}

\newtheorem{theorem}{Theorem}%[section]
\newtheorem{proposition}{Proposition}[section]
\newtheorem{lemma}{Lemma}[section]

\newtheorem{remark}{Remark}[section]
\numberwithin{equation}{section}

\begin{document}

\begin{frontmatter}
 \title{Unfitted boundary algebraic equation method based on difference potentials and lattice Green's function in 3D}

\author{Qing~Xia\corref{cor}}
% \author[wku]{Qing~Xia\corref{cor}}
\ead{qxia@kean.edu}

\address{Department of Mathematics, Wenzhou Kean University, Zhejiang, China, 325060.}

\begin{abstract}

This work presents an unfitted boundary algebraic equation (BAE) method for solving three-dimensional elliptic partial differential equations on complex geometries using finite difference on structured meshes.  We demonstrate that replacing finite auxiliary domains with free-space LGFs streamlines the computation of difference potentials, enabling matrix-free implementations and significant cost reductions. We establish theoretical foundations by showing the equivalence between direct formulations in difference potentials framework and indirect single/double layer formulations and analyzing their spectral properties. The spectral analysis demonstrates that discrete double layer formulations provide better-conditioned systems for iterative solvers, similarly as in boundary integral method. The method is validated through matrix-free numerical experiments on both Poisson and modified Helmholtz equations in 3D implicitly defined geometries, showing optimal convergence rates and computational efficiency. This framework naturally extends to unbounded domains and provides a foundation for applications to more complex systems like Helmholtz and Stokes equations.

\end{abstract}


\end{frontmatter}

% ------------- Table of contents
% \tableofcontents

% \clearpage 

% -------------------------------------------------------------------------------------------------------------------------

\input{tex/intro}
\input{tex/dpm}
\input{tex/bae}
\input{tex/numerics}
\input{tex/conclusion}

\section*{Acknowledgement}
This work is partially funded by Natural Science Foundation of China (NSFC Grant No: 12401546) and Wenzhou Kean University (Grant No: ISRG2024003).
% -------------------------------------------------------------------------------------------------------------------------

% ========================= APPENDIX =========================
% \appendix

% % -------------------------------------------------------------------------------------------------------------------------
% \input{tex/appendix}

% -------------------------------------------------------------------------------------------------------------------------
% \clearpage

\bibliographystyle{elsart-num}
\bibliography{references.bib}
 
\end{document}

%% file: tex/intro.tex
\section{Introduction}

Elliptic partial differential equations (PDEs) form the cornerstone of numerous scientific and engineering applications, from steady-state heat conduction to electrostatics and fluid dynamics. These equations also arise naturally when implicit time-stepping schemes are employed for parabolic or hyperbolic problems, where stability and accuracy requirements often necessitate the solution of elliptic systems at each time step. A fundamental challenge in solving these equations lies in handling complex geometries while maintaining computational efficiency and numerical accuracy.

Traditional numerical approaches for elliptic PDEs broadly fall into three categories:
\begin{itemize}
\item Body-fitted mesh methods (finite elements, boundary-fitted finite differences) \cite{barrett1987fitted} offer high accuracy but require complex mesh generation and management, particularly challenging for moving boundaries or topology changes.
\item Meshless methods (radial basis functions \cite{fornberg2015solving}, boundary integral methods \cite{mckenney1995fast,zhong2018implicit,zhou2024correction,ying2013kernel}, spectral method \cite{gu2021efficient}) provide geometric flexibility and often lead to dense linear systems and face challenges in handling variable coefficients.
\item Embedded boundary methods such as immersed finite difference/element \cite{peskin2002immersed,leveque1994immersed}, cut-cell methods \cite{chen2023arbitrarily,chen2021adaptive,hansbo2002unfitted,hansbo2014cut}, Difference Potentials Method (DPM) \cite{ryaben2012method}, combine the efficiency of structured grids with geometric flexibility but might struggle with accuracy near boundaries and stability issues.
\end{itemize}

Among these approaches, boundary integral methods (BIM) deserve special attention as they share important similarities with our proposed approach. BIM reduce volumetric PDEs to boundary integral equations, achieving dimension reduction and naturally handling unbounded domains. These methods excel in problems with constant coefficients where fundamental solutions are known analytically, and they provide high accuracy by incorporating the exact representation of solutions through Green's functions.
Our proposed approach shares the dimensional reduction advantage of BIM but operates in a discrete setting using lattice Green's functions instead of continuous ones. This discrete framework offers several distinct advantages: it naturally aligns with computational grids, avoids singular integral evaluations common in BIM, and maintains the ability to handle variable coefficients through the underlying difference potentials framework. 

Difference Potentials Method (DPM) \cite{ryaben2012method} offers a general framework of solving elliptic equations in arbitrary geometry using unfitted structured meshes. By the introduction of auxiliary domains and the discrete Green's function defined on the auxiliary domain, the discrete elliptic equations in the original domain can be reduced to boundary algebraic equations at grid points near the continuous boundary.  However, existing implementations face significant computational challenges, particularly in three dimensions. The primary bottleneck lies in constructing difference potentials operators, which traditionally requires multiple volume solves in the auxiliary domain. Typically, fast Poisson solvers such as those relied on FFTs are employed for acceleration of those multiple volume solves, which we aim to bypass in this work.

This work builds on previous work of using local basis function derived the Galerkin difference method \cite{banks2016galerkin} to construct difference potentials operators \cite{xia2023local}. The main novelty is the integration lattice Green's functions (LGFs) into the DPM framework. LGFs, which serve as fundamental solutions for discrete elliptic operators on infinite lattices, enable explicit representation of boundary potentials without relying on solving auxiliary problems defined in finite domains. This integration provides several key advantages, including reduction in computational complexity, matrix-free implementation capabilities for better memory efficiency and natural handling of unbounded domains.

Lattice Green's functions have been studied in many applications such as in \cite{liska2016fast,dorschner2020fast,liska2014parallel,cserti2000application,joyce2004exact,borwein2013lattice}.  In this work, we focus on Cartesian lattices, while triangular, honeycomb lattices are also possible \cite{zhao2011extension}. While lattice Green's functions have been previously studied for simple geometries and infinite domains, their systematic integration with difference potentials for complex geometries represents a significant advancement. 

This idea of using boundary algebraic equations to solve elliptic partial differential equations numerically can be traced back to a NASA technical report \cite{saltzer1958discrete}. Our approach builds on difference potentials and earlier work by Martinsson et al on boundary algebraic equations (BAE) \cite{martinsson2009boundary}. The key idea of BAE is to mimic its continuous counterpart, the boundary integral method, where the fundamental solution of the discrete Poisson equation is taken as the lattice Green's function. Discrete single and double layer potentials are defined similarly as in the boundary integral method, where a system of boundary algebraic equations are solved for unknown density at the boundary lattice points. One major difference is that no singularity occur in the BAE when the source point coincides with the target. Interior or exterior solutions at any point then can be obtained by summation of the source-target interactions, using discrete layer potentials. 

In the current work, we do not require grid nodes to align with boundary curves, hence the method is termed as ``unfitted'' BAE. We extend the BAE framework to handle arbitrary geometries through a rigorous theoretical foundation. In this work, BAE is also capable of handling nonhomogeneous source terms in finite domains via Fast Fourier Transform. For infinite domains and nonhomogeneous source functions, fast summation techniques such as fast multipole method \cite{beatson1997short} will be needed and studied in future work.

Our method establishes equivalence of boundary equations in DPM and boundary equations in BAE, so as to construct boundary equations directly without relying on auxiliary volume solves. Such direct formulation is also suitable for matrix-free implementation, offering memory efficiency for large-scale computations. Inheriting merits of both DPM and BAE, our method naturally admits property of dimension reduction, and offers a robust foundation for multiphysics simulations in complex domains.  The framework naturally extends to unbounded domains and provides pathways for applications to more challenging systems such as high-frequency Helmholtz problems, Stokes flows and etc.

The rest of the manuscript is organized as follows: Section \ref{sec:dpm} presents the fundamental theory of difference potentials and their extension to infinite domains using LGFs. Section \ref{sec:bae} develops direct and indirect boundary algebraic equation formulations, analyzing their spectral properties and establishing equivalence relationships. Section \ref{sec:numerical_results} validates the method through numerical experiments for Poisson and modified Helmholtz equations in 3D implicitly defined geometries. Section \ref{sec:conclusion} concludes with future directions and potential applications.

%% file: tex/dpm.tex
\section{Difference Potentials Method}\label{sec:dpm}

In this work, we focus on efficient and accurate numerical solution of the elliptic partial differential equations in the form of:
\begin{subequations}\label{eqn:elliptic}
\begin{align}
-\Delta u +\sigma u &= f,\quad x\in\Omega\label{eqn:pde}\\
u &= g,\quad x\in\partial\Omega\label{eqn:pde_bc}
\end{align}
\end{subequations}
with $\sigma\geq0$. The geometry of $\Omega\subset\mathbb{R}^3$ is arbitrarily shaped and the PDE in \eqref{eqn:elliptic} is assumed well-posed. For time-dependent problems, when implicit time stepping is employed, $\sigma$ will be $1/\Delta t$ or $1/\Delta t^2$ where $\Delta t$ denotes the time stepping size. In this work, we will demonstrate the developed numerical algorithm with Dirichlet boundary condition \eqref{eqn:pde_bc}. Other types of boundary conditions can be handled similarly as discussed in \cite{xia2023local}.

\subsection{Difference Potentials in the finite domain}

Given a bounded domain $\Omega\subset \mathbb{R}^3$ of arbitrary shapes, we first embed it into a computationally simple domain $\Omega^0$ such as a cube. Then $\Omega^0$ is discretized using a Cartesian mesh. The grid points will be denoted: $M^0$ for points in $\Omega^0$, $M^+$ for points in $\Omega$, $M^-:=M^0\backslash M^+$. The stencil points $N^\pm$ are defined as all points needed in the formulation of a finite difference operator for points in $M^\pm$, respectively, and $N^0:=N^+\cup N^-$. In the second order case in 3D, $N^\pm$ consist of the seven-point stencil:
\begin{align}
N^\pm = \left\{(x_{i\pm 1},y_j,z_k),(x_i,y_{j\pm1},z_k),(x_i,y_j,z_{k\pm1}),(x_i,y_j,z_k) \ \big \vert\  (x_i,y_j,z_k)\in M^\pm\right\}.
\end{align}
The point sets $N^\pm$ have an intersection $\gamma:=N^+\cap N^-$ known as \emph{the discrete grid boundary} and $\gamma$ can be divided into two subsets, where $\gamma_+$ corresponds to the interior of $\Omega$ and $\gamma_-$ the exterior of $\Omega$. See Figure~\ref{fig:points} for an illustration of these two point sets $\gamma^\pm$ in 2D.

\begin{figure}[htbp]
\centering
\includegraphics[width=0.4\textwidth]{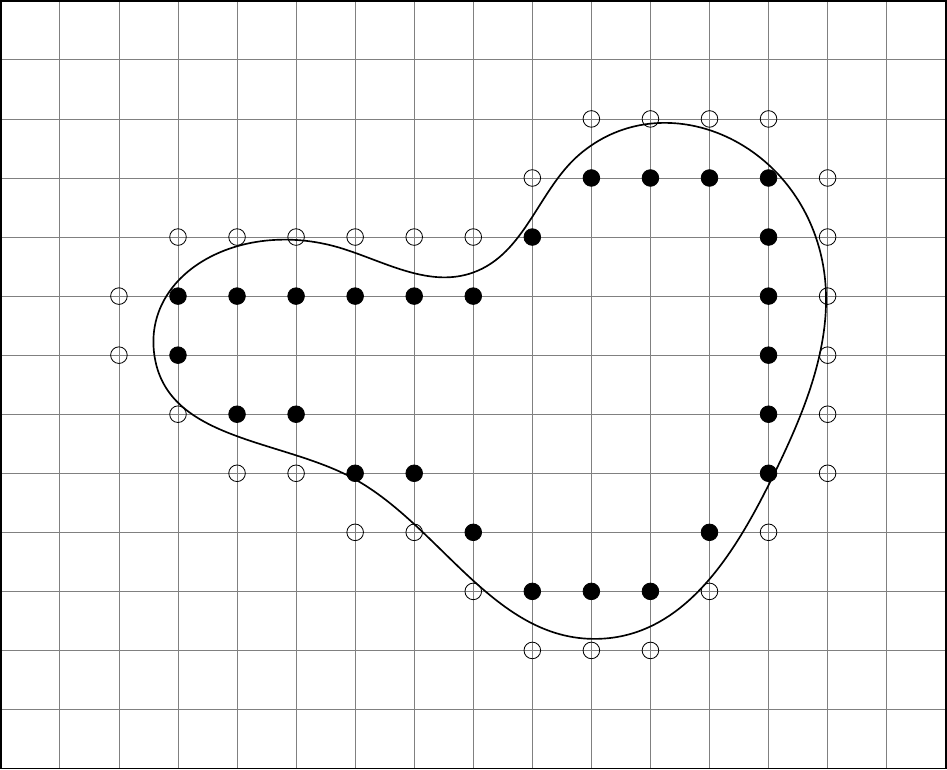}
\caption{Example of point sets in 2D (dot: $\gamma_{+}$, circle: $\gamma_-$)}\label{fig:points}
\end{figure}

With these point sets defined, we can discretize the PDE \eqref{eqn:pde} on point set $M^+\subset\Omega$
\begin{align}\label{eqn:discrete}
L_hu_{jkl}:= -\Delta_h u_{jkl}+\sigma u_{jkl} = f_{jkl},\quad x_{jkl}\in M^+
\end{align}
where $\Delta_h$ will be taken as the 7-point central finite difference stencil.
Note that values at $\gamma_-$ is unknown and boundary conditions need to be imposed for boundary closures, which will be discussed in details later. Given the discretization \eqref{eqn:discrete}, we introduce the auxiliary problem defined on the point set $N^0$
\begin{subequations}\label{eqn:aux}
\begin{align}
L_h u_{jkl} &= q_{jkl},\quad x_{jkl}\in M^0\\
u_{jkl} &= 0,\quad x_{jkl}\in N^0\backslash M^0
\end{align}
\end{subequations}
where $q_{jkl}$ is some grid function defined on $M^0$. We will denote the inverse operation of finding solutions to the difference equations \eqref{eqn:aux} with the given boundary condition as $u_{jkl}=G_hq_{jkl}$.

For a finite and computationally simple auxiliary domain $\Omega^0$, the operator $G_h$ is taken to be some fast inversion of the $L_h$. For example, FFT can be used to solve $L_h u_{jkl}=q_{jkl}$ for grid function $q_{jkl}$ and no explicit knowledge of the Green's function for the auxiliary domain $\Omega^0$ is needed. Now we will define the particular solution obtained with source function $q_{jkl}=\chi_{M^+}f_{jkl}$:
\begin{align}\label{eqn:ps}
G_hf_{jkl} = G_h[\chi_{M^+}f_{jkl}],\quad x_{jkl}\in M^0
\end{align}
where $\chi_{M^+}$ denotes a characteristic function for set $M^+$, and the difference potentials obtained with source function $q_{jkl}=\chi_{M^-}L_hw_{jkl}$
\begin{align}\label{eqn:dp}
P_{N^+\gamma}w_{jkl}=G_h[\chi_{M^-}L_hw_{jkl}]
\end{align}
where $w_{jkl}$ is any grid function defined on $N^+$.

It can be checked that the superposition of the particular solution \eqref{eqn:ps} and the difference potential \eqref{eqn:dp} satisfies the difference equation \eqref{eqn:discrete} in $M^+$.

The following key theorem establishes the dimension reduction relation from the volumetric difference equation \eqref{eqn:discrete} to algebraic equations at grid points only near the boundary, which can be found in many previous work, e.g. \cite{xia2023local,ryaben2006algorithm,medvinsky2012method,albright2015high,epshteyn2014algorithms,epshteyn2012upwind,epshteyn2019efficient}.
\begin{theorem}
The discrete equation \eqref{eqn:discrete} for points in $M^+$ is equivalent to the following boundary equation with projections:
\begin{align}\label{eqn:bep}
u_\gamma - P_\gamma u_\gamma = G_hf_\gamma.
\end{align}
where $P_\gamma u_\gamma:=Tr_{\gamma}G_h[\chi_{M^-}L_hu_\gamma]$ and $G_hf_\gamma:=Tr_\gamma G_h[\chi_{M^+}f_h]$ and $Tr_{\gamma}$ denotes the trace operator on the set $\gamma$, $\chi_{M^-}$ denotes the characteristic function for the set $M^-$.
\end{theorem}

Each column of the operator $P_\gamma$ can be constructed using a unit density
\begin{align}\label{eqn:unit_density}
u_{j^*k^*l^*}(j,k,l) = \left\{
\begin{array}{cc}
1,& j=j^*,k=k^*,l=l^*,\\
0,&\mbox{ elsewhere.}
\end{array}
\right.
\end{align}
and the corresponding column in $P_\gamma$ for point $x_{j^*k^*l^*}$ in the set $\gamma$ is thus
\begin{align}
P_\gamma Tr_{\gamma}u_{j^*k^*l^*}=Tr_{\gamma}G_h[\chi_{M^-}L_hu_{j^*k^*l^*}].
\end{align}
Essentially we are solving the auxiliary problem
\begin{subequations}\label{eqn:unit_finite_aux_prob}
\begin{align}
L_h u_{jkl} &= \chi_{M^-}L_hu_{j^*k^*l^*}\\
u_{jkl} &= 0 
\end{align}
\end{subequations}
and inject $Tr_{\gamma}u_{jkl}$ as the corresponding column for $P_\gamma$.

The boundary equations with projections \eqref{eqn:bep} can be further reduced to the interior set $\gamma_+$ only, as illustrated by the following theorem.
\begin{theorem}
The boundary equation with projections \eqref{eqn:bep} defined on $\gamma$ is equivalent to the following reduced boundary equation with projections \eqref{eqn:rbep} defined on $\gamma_+$.
\begin{align}\label{eqn:rbep}
u_{\gamma_+} - P_{\gamma_+} u_\gamma = G_hf_{\gamma_+}.
\end{align}
where $P_{\gamma_+} u_\gamma:=Tr_{\gamma_+}G_h\chi_{M^-}L_h[u_\gamma]$ and $G_hf_{\gamma_+}:=Tr_{\gamma_+} G_h[\chi_{M^+}f_h]$.
\end{theorem}
The proof can be found, for example, in \cite{epshteyn2019efficient}.

The reduced boundary equations with projections \eqref{eqn:rbep} is equivalent to the finite difference equations \eqref{eqn:discrete}. To close the linear system, one can use the local basis functions approach developed in \cite{xia2023local}. For example, the boundary condition can be discretized as
\begin{align}\label{eqn:bc}
\sum_{x_{jkl}\in\gamma} u_{jkl} \Phi_{jkl}(x^*,y^*,z^*) = g(x^*,y^*,z^*)
\end{align}
in the case of Dirichlet boundary condition, 
where $(x^*,y^*,z^*)$ is a boundary point in a cut cell and $\Phi_{jkl}$ is a nodal basis function defined at point $x_{jkl}$. In the second order case, we can take $\Phi_{jkl}$ to be the standard hat function in 3D, while higher order version will take the form of difference Galerkin basis function studied in \cite{jacangelo2020galerkin,banks2016galerkin}.

Now we have a closed square system with $|\gamma|$ number of equations and $|\gamma|$ number of unknowns, where $|\gamma|$ is the cardinality of point set $\gamma$.
\begin{subequations}\label{eqn:inhom}
\begin{align}
u_{\gamma_+} - P_{\gamma_+} u_\gamma &= G_hf_{\gamma_+}\\
\sum_{x_{jkl}\in\gamma} u_{jkl} \Phi_{jkl}(x^*,y^*,z^*) &= g(x^*,y^*,z^*)
\end{align}
\end{subequations}

The following lemma will serve to reduce the nonhomogeneous boundary equations \eqref{eqn:inhom} to homogeneous ones.

\begin{theorem}\label{lem:dp_gh}
The trace of the difference potential with density of the particular solution $G_hf_\gamma$ is 0, i.e.
\begin{align}
P_\gamma [G_hf_\gamma] = 0.
\end{align}
\end{theorem}

\begin{proof}
The difference potential for any density $u_\gamma$ admits the following definition
\begin{align}
P_{N^+\gamma}u_\gamma:=G_h \chi_{M^-}L_hu,\quad x_{jkl}\in N^+
\end{align}
where $u$ is some extension of the density $u_\gamma$ to the grid $N$, i.e. $Tr_{\gamma}u = u_\gamma$. It can be argued that the exact form of extension does not change the value of the difference potential. To see this, assume we have $u_1\neq u_2$ but $Tr_\gamma u_1=Tr_{\gamma} u_2 = u_\gamma$ then in $N^+$ and define $w=u_1-u_2$, then $Tr_{\gamma}w=0$, and
\begin{align}
L_hw = L_h(u_1-u_2) = 0,\quad x_{jkl}\in M^+
\end{align}
due to the restriction operator $\chi_{M^-}$. The above linear system admits only zero solution in $M^+$ as the boundary is zero. For $M^-$, similar arguments also show $w=0$ in $M^-$.
Hence the difference potential is the same no matter how the extension is formed. Other proofs of this result can be found in Ryabenkii's book as well.

Now for $P_\gamma [G_hf_\gamma]$, we will choose the particular solution as the extension $[G_h\chi_{M^+}f]$ then the difference potential for density $[G_hf_\gamma]$ will be computed as
\begin{align}
L_h[P_{N^+\gamma}G_hf_\gamma] = \chi_{M^-}L_h[G_h\chi_{M^+}f_h]=0
\end{align}
since $M^+$ and $M^-$ are dis-adjoint. When homogeneous BC are imposed for the auxiliary problem, the difference potential of the particular solution will be 0.
\end{proof}

Now we define the homogeneous density
\begin{align}\label{eqn:vgamma}
v_\gamma = u_\gamma - G_hf_\gamma
\end{align} 
and the following proposition holds.

\begin{proposition}\label{prop:vgamma}
The reduced boundary equation with projections \eqref{eqn:rbep} is equivalent to the boundary equations for $v_{\gamma}$ defined in \eqref{eqn:vgamma}:
\begin{align}\label{eqn:beq}
v_{\gamma_+} - P_{\gamma_+} v_\gamma = 0
\end{align}
\end{proposition}
\begin{proof}
``$\Rightarrow$'': 
By definition of $v_\gamma:=u_\gamma - G_hf_\gamma$, $u_\gamma=v_\gamma+G_hf_\gamma$, then the reduced boundary equation gives
\begin{subequations}
\begin{align}
u_{\gamma_+} - P_{\gamma_+} u_\gamma &= G_hf_{\gamma_+}\\
\Rightarrow v_{\gamma_+}+G_hf_{\gamma_+} - P_{\gamma_+} (v_\gamma+G_hf_\gamma) &= G_hf_{\gamma_+}\\
\Rightarrow v_{\gamma_+} - P_{\gamma_+} (v_\gamma+G_hf_\gamma) &= 0\\
\Rightarrow v_{\gamma_+} - P_{\gamma_+} v_\gamma &= P_{\gamma_+}G_hf_\gamma
\end{align}
\end{subequations}
The last equation is by Lemma~\ref{lem:dp_gh}.

``$\Leftarrow$'': Similarly, from the definition of $v_\gamma$,
\begin{subequations}
\begin{align}
v_{\gamma_+} - P_{\gamma_+} v_\gamma &= 0\\
\Rightarrow u_{\gamma_+} - G_hf_{\gamma_+} - P_{\gamma_+} [u_\gamma - G_hf_\gamma] &=0\\
\Rightarrow u_{\gamma_+} - G_hf_{\gamma_+} - P_{\gamma_+} u_\gamma &=0\\
\Rightarrow u_{\gamma_+} - P_{\gamma_+} u_\gamma &=G_hf_{\gamma_+}
\end{align}
due to Lemma~\ref{lem:dp_gh} $P_{\gamma_+}G_hf_\gamma=0$ as well.
\end{subequations}
\end{proof}

By Proposition~\ref{prop:vgamma}, the boundary systems \eqref{eqn:inhom} is equivalent to
\begin{subequations}\label{eqn:hom}
\begin{align}
v_{\gamma_+} - P_{\gamma_+} v_\gamma &= 0\\
\sum_{x_{jkl}\in\gamma} v_{jkl} \Phi_{jkl}(x^*,y^*,z^*) &= g(x^*,y^*,z^*)-\sum_{x_{jkl}\in\gamma}[G_hf_{\gamma}]_{jkl} \Phi_{jkl}(x^*,y^*,z^*)
\end{align}
\end{subequations}
where we can solve $v_\gamma$. Similar approach has been widely adopted in the boundary integral equation method for nonhomogeneous elliptic equations, see for example \cite{mayo1984fast}.

Once $v_\gamma$ is obtained, the density $u_\gamma$ is given by
\begin{align}
u_\gamma = v_\gamma + G_hf_\gamma,
\end{align}
and the approximation of the numerical solution to \eqref{eqn:elliptic} is given by the following discrete generalized Green's formula:
\begin{align}\label{eqn:discrete_green_u}
u_{jkl} = G_h[\chi_{M^-}L_hu_\gamma] + G_h[\chi_{M^+}f_h], \quad x_{jkl}\in N^+
\end{align}

\begin{remark}
Due to Lemma~\ref{lem:dp_gh}, we can simply replace the inhomogeneous density $u_\gamma$ with the homogeneous density $v_{\gamma}$ in the Green's formula.
\begin{align}\label{eqn:discrete_green_v}
u_{jkl} = G_h[\chi_{M^-}L_hv_\gamma] + G_h[\chi_{M^+}f_h], \quad x_{jkl}\in N^+,
\end{align}
which basically implies that we only need to solve two linear systems: one homogeneous system that accounts for the boundary data and one for the nonhomogeneous source functions.
\end{remark}

The most computationally expensive part in the above approximation of the discrete Generalized Green's formula lies in the construction of the density $u_\gamma$, which relies on the operator $P_{\gamma_+}$. Different approaches have been developed to obtain the operator, such as combining extension operators of Taylor's form with spectral method on the continuous boundary \cite{ryaben2012method,medvinsky2012method,epshteyn2012upwind}, etc., and method based on the Galerkin difference basis functions \cite{xia2023local}.

In \cite{xia2023local}, we demonstrated how to combine local Galerkin difference basis functions with difference potentials method and argued that the computational cost for constructing $P_{\gamma}$ or $P_{\gamma_+}$ would scale like a formidable $\mathcal{O}(N^5\log(N))$ in 3D, since we need to compute the difference potentials at $|\gamma|\sim\mathcal{O}(N^2)$ points in the set $\gamma$ using unit density defined at each point in $\gamma$. Here $N$ is the number of unknowns in either $x-$, $y-$ or $z-$ direction in the uniform mesh for the auxiliary domain. The computation of each difference potential corresponding to the unit density would be $\mathcal{N^3}\log(N)$ if FFT is employed to solve the auxiliary problem. In this work, we will continue using the Galerkin difference basis function and show how to avoid the computational cost of constructing $P_{\gamma}$ or $P_{\gamma_+}$ and how to obtain the operators directly from lattice Green's functions.

\subsection{Difference Potentials in the infinite domain}

In this subsection, we will focus on the homogeneous equation as we discussed in the previous subsection:
\begin{align}
Lv:=-\Delta v+\sigma v = 0, \quad x\in \Omega
\end{align}
in the same domain $\Omega$ of arbitrary shapes and with the same boundary condition as in $\eqref{eqn:discrete}$.

Instead of using the finite auxiliary domain $\Omega^0$, we now turn to the free space $\mathbb{R}^3$ as the auxiliary domain and introduce the infinite lattice ${\mathbb{Z}h}^3$ with spacing $h$. Here we assume that the infinite grid coincide with the finite grid on $M^+$. The point sets $N^\pm, M^\pm$ are defined similarly as in the finite auxiliary domain. The point set $\gamma$, $\gamma_\pm$ will be identically defined. To construct the difference potential operator $\widehat{P}_{\gamma}$ and $\widehat{P}_{\gamma_+}$, we use similar auxiliary problem as in \eqref{eqn:unit_finite_aux_prob} and use the unit density $v_{j^*,k^*,l^*}$. In other words, we solve 
\begin{subequations}\label{eqn:unit_infinite_aux_prob}
\begin{align}
\widehat{L}_h \hat{v}_{jkl} &= \chi_{M^-}\widehat{L}_h\hat{v}_{j^*k^*l^*},\\
\lim_{jkl\rightarrow\infty}v_{jkl} &= 0,
\end{align}
\end{subequations} 
and $\hat{v}_{jkl}=\widehat{G}_h\chi_{M^-}\widehat{L}_h\hat{v}_{j^*k^*l^*}$,
where the ``hat'' denotes that the operator is obtained in the free space. The fundamental solution to \eqref{eqn:unit_infinite_aux_prob} is commonly known as the lattice Green's function. In the case of infinite lattices, we will denote the lattice Green's function as $G_{\sigma,h}(jkl)$ and it satisfies
\begin{align}\label{eqn:lgf_delta}
\widehat{L}_h\widehat{G}_{\sigma,h}(jkl) = \delta_{jkl}
\end{align}
where the discrete delta function $\delta_{jkl}$ is 1 at point $x_{jkl}$ and 0 elsewhere. 

For constant-coefficient linear operators, the second order central finite difference scheme would give the discrete Green's functions in terms of a triple integral when one performs Fourier transform on the infinite lattice and the inverse Fourier transform will give:
\begin{align}\label{eqn:lgf}
\widehat{G}_{\sigma,h}(j,k,l) = \frac{h^2}{(2\pi)^3}\int_{-\pi}^{\pi}\int_{-\pi}^{\pi}\int_{-\pi}^{\pi} \frac{\cos(j\theta_1)\cos(k\theta_2)\cos(l\theta_3)}{6+\sigma h^2  - 2(\cos\theta_1+\cos\theta_2+\cos\theta_3)}d\theta_1d\theta_2d\theta_3
\end{align}

We will focus on the case of $h=1$, as the division by $h^2$ will cancel out in the left and right hand sides of \eqref{eqn:unit_infinite_aux_prob}. Thus we will drop the hat and $h$ in $\widehat{G}_{\sigma,h}(j,k,l)$, and use $G_{\sigma}(j,k,l)$ to denote the lattice Green's function, without risk of confusion. 

For $\sigma=0$, $G_{0}(j,k,l)$ reduces to the classical lattice Green's function and the well-known Watson's integral \cite{joyce2005evaluation,zucker201170}. The evaluation of different types of Watson's integral has been well-studied. For example, $G_{0}(0,0,0)$ is analytically computed as
\begin{align}
G_{0}(0,0,0) = \frac{\sqrt{6}}{96\pi^3}\Gamma\left(\frac{1}{24}\right)\Gamma\left(\frac{5}{24}\right)\Gamma\left(\frac{7}{24}\right)\Gamma\left(\frac{11}{24}\right)
\end{align}
or 
\begin{align}
G_{0}(0,0,0) = \frac{\sqrt{3}-1}{96\pi^3}\left[\Gamma\left(\frac{1}{24}\right)\Gamma\left(\frac{11}{24}\right)\right]^2
\end{align}
as in \cite{zucker201170}. Recursive formulas thus can be derived based on the definition \eqref{eqn:lgf_delta} and the exact value of $G_{0}(0,0,0)$.

Higher-order or non-standard finite difference stencils would give similar expressions of the lattice Green's function in \eqref{eqn:lgf}.  Evaluations of high order versions of the LGF \eqref{eqn:lgf} resulted from various forms of finite difference schemes have been studied in a recent work \cite{gabbard2024lattice}. The algorithm in this work can be extended naturally to high order accuracy once high order lattice Green's functions are obtained.

We will mainly focus on two cases: $\sigma=0$ and $\sigma>0$. 
\begin{itemize}
	\item 
For $\sigma=0$, the integrand in \eqref{eqn:lgf} is singular at the origin, but can be expressed in terms of an integral of Bessel's functions (Watson's transformation)
\begin{align}\label{eqn:lgf_bessel}
G_{0}(j,k,l)=\int_0^\infty\exp(-6t)I_j(2t)I_k(2t)I_l(2t)dt
\end{align}
where 
\begin{align}
I_n(t)  = \frac{1}{2\pi}\int_{-\pi}^{\pi} \exp(t\cos\theta)\cos(n\theta)d\theta
\end{align}
The improper integral \eqref{eqn:lgf_bessel} can be split into two integrals with a large enough $T^*$
\begin{align}
G_{0}(l,m,n) = \int_0^{T^*}\exp(-6t)I_l(2t)I_m(2t)I_n(2t)dt+\int_{T^*}^\infty\exp(-6t)I_l(2t)I_m(2t)I_n(2t)dt
\end{align}
where the first finite integral can be evaluated using quadrature rules and the second integral using asymptotic expansions of the Bessel functions. The asymptotic expansions for large argument $|t|$ and $\mu=4n^2$ (\cite[p375-377]{abramowitz1948handbook}):
\begin{align}
I_n(t) \sim \frac{\exp(t)}{\sqrt{2\pi t}}\left(1-\frac{\mu-1}{8t}+\frac{(\mu-1)(\mu-9)}{2!(8t)^2}-\frac{(\mu-1)(\mu-9)(\mu-25)}{3!(8t)^3}+\cdots\right)
\end{align}

The improper integral \eqref{eqn:lgf_bessel} can also be evaluated using adaptive Gauss Kronrod method. For large magnitude of $(l,m,n)$, asymptotic formula have been established in \cite{martinsson2002asymptotic}. The $q$-term asymptotic expansion of $G_{0}(\bm{n})$ with $\bm{n}=(j,k,l)$ is given by
\begin{align}
G_{0}(\bm{n}) = A_G^q(\bm{n})+\mathcal{O}(|n|^{-2q-1})
\end{align}
For large enough $\bm{n}$, we can take $q=2$:
\begin{align}
A^2_G(\bm{n})=-\frac{1}{4\pi|\bm{n}|}-\frac{j^4+k^4+l^4-3j^2k^2-3j^2l^2-3k^2l^2}{16\pi|\bm{n}|^7}
\end{align}
or $q=3$:
\begin{align}
A_G^3(\bm{n}) = A^2_G(\bm{n})&+\frac{1}{128\pi|\bm{n}|^{13}}\Big(-288(k^2l^2j^4+k^2l^4j^2+k^4l^2j^2)+621(j^4k^4+l^4k^4+j^4l^4)\nonumber\\
&-244(k^2j^6+l^2j^6+k^6j^2+l^6j^2+k^2l^6+k^6l^2)+23(j^8+k^8+l^8)\Big)
\end{align}
As $|\bm{n}|\rightarrow\infty$, $G_{0}(\bm{n})$ behaves like the fundamental solution of the continuous Poisson's equation $-1/4\pi |x|$.

\item For $\sigma>0$, the LGF \eqref{eqn:lgf} is well defined and the spectral accuracy of trapezoidal rule for periodic functions can be utilized:
\begin{subequations}
\begin{align}
G_{\sigma}(j,k,l)=\frac{h^2}{(2\pi)^3}\int_{-\pi}^{\pi}\int_{-\pi}^{\pi}\int_{-\pi}^{\pi} \frac{e^{ij\theta_1}e^{ik\theta_2}e^{il\theta_3}}{6+\sigma h^2  - 2(\cos\theta_1+\cos\theta_2+\cos\theta_3)}d\theta_1d\theta_2d\theta_3\\
\approx\frac{1}{N^3}\sum_{p=-N/2}^{\frac{N}{2}-1}\sum_{q=-\frac{N}{2}}^{\frac{N}{2}-1}\sum_{r=-\frac{N}{2}}^{\frac{N}{2}-1} \frac{e^{ij2\pi p/N}e^{ik2\pi q/N}e^{il2\pi r/N}}{6+\sigma h^2  - 2(\cos(2\pi p/N)+\cos(2\pi q/N)+\cos(2\pi r/N))}
\end{align}
\end{subequations}
Thus 3D FFTs can be employed to speed up evaluations of source-target interactions at multiple instances. For large enough $\sigma$, the values of $G_\sigma$ decays rapidly as $|\bm{n}|$ tends to infinity.

\end{itemize}

With lattice Green's function defined, we are able to solve \eqref{eqn:unit_infinite_aux_prob} with ease.
Notice that the right hand side in the difference equation is highly localized with at most six nonzero entries. Then the solution $\hat{v}_{jkl}$ can be explicitly written as
\begin{align}
\hat{v}_{jkl} =& (-6+\sigma h^2)\chi_{M^-}(p,q,r)G_0(p,q,r)\nonumber\\
&+\chi_{M^-}(p+1,q,r)G_0(p+1,q,r) +\chi_{M^-}(p-1,q,r)G_0(p-1,q,r) \nonumber\\
&+\chi_{M^-}(p,q+1,r)G_0(p,q+1,r) +\chi_{M^-}(p,q-1,r)G_0(p,q-1,r) \nonumber\\
&+\chi_{M^-}(p,q,r+1)G_0(p,q,r+1) +\chi_{M^-}(p,q,r-1)G_0(p,q,r-1) 
\end{align}
for any $x_{jkl}\in \gamma$ with unit density defined at $x_{pqr}\in\gamma$, thus for columns in $\widehat{P}_\gamma$. This significantly reduces the computational cost of the construction of $\widehat{P}_\gamma$ from $\mathcal{O}(N^5\log(N))$ to $\mathcal{O}(N^2)$ when unit density is used.

Likewise, we have the following theorem for homogeneous equations:
\begin{theorem}
The difference equation $\widehat{L}_h\hat{v}_{jkl}=0$ for $x_{jkl}\in M^+$ is equivalent to the boundary equation with projections
\begin{align}\label{eqn:inf_veq}
\hat{v}_{\gamma_+}-\widehat{P}_{\gamma_+}\hat{v}_{\gamma} = 0,
\end{align}
\end{theorem}
which we will not prove as this is a classical result in \cite{ryaben2012method}.

Replacing the boundary equations with projections in \eqref{eqn:hom} with \eqref{eqn:inf_veq} and $v_\gamma$ with $\hat{v}_\gamma$, we arrive at the following system of equations:
\begin{subequations}\label{eqn:inf_hom}
\begin{align}
\hat{v}_{\gamma_+} - \widehat{P}_{\gamma_+} \hat{v}_\gamma &= 0\\
\sum_{x_{jkl}\in\gamma} \hat{v}_{jkl} \Phi_{jkl}(x^*,y^*,z^*) &= g(x^*,y^*,z^*)-\sum_{x_{jkl}\in\gamma}[G_hf_{\gamma}]_{jkl} \Phi_{jkl}(x^*,y^*,z^*)
\end{align}
\end{subequations}
The difference between \eqref{eqn:hom} and \eqref{eqn:inf_hom} lies in the difference potentials operator $\widehat{P}_{\gamma_+}$ and $P_{\gamma_+}$, where the former with ``hat'' can be obtained with much less computational cost.

Now we establish the equivalence between \eqref{eqn:hom} and \eqref{eqn:inf_hom}.

\begin{theorem}
The densities $v_\gamma$ and $\hat{v}_\gamma$ from solving \eqref{eqn:hom} and \eqref{eqn:inf_hom} are identical, if the meshes align.
\end{theorem}

\begin{proof}
$v_{\gamma_+} - P_{\gamma_+} \hat{v}_\gamma = 0$ is equivalent to $L_hv_{jkl}=0$ for $x_{jkl}\in M^+$, thus \eqref{eqn:hom} is equivalent to
\begin{subequations}
\begin{align}
L_hv_{jkl}&=0,\quad x\in M^+\\
\sum_{x_{jkl}\in\gamma} v_{jkl} \Phi_{jkl}(x^*,y^*,z^*) &= g(x^*,y^*,z^*)-\sum_{x_{jkl}\in\gamma}[G_hf_{\gamma}]_{jkl} \Phi_{jkl}(x^*,y^*,z^*)
\end{align}
\end{subequations}

and $\hat{v}_{\gamma_+} - \widehat{P}_{\gamma_+} \hat{v}_\gamma = 0$ is equivalent to  $\widehat{L}_h\hat{v}_{jkl}=0$ for $x_{jkl}\in M^+$, thus  \eqref{eqn:inf_hom} is equivalent to 
\begin{subequations}
\begin{align}
\widehat{L}_h\hat{v}_{jkl} &= 0,\quad x\in M^+\\
\sum_{x_{jkl}\in\gamma} \hat{v}_{jkl} \Phi_{jkl}(x^*,y^*,z^*) &= g(x^*,y^*,z^*)-\sum_{x_{jkl}\in\gamma}[G_hf_{\gamma}]_{jkl} \Phi_{jkl}(x^*,y^*,z^*)
\end{align}
\end{subequations}
Notice that $L_h=\widehat{L}_h$ in $M^+$, the difference of $v$ and $\hat{v}$ satisfies 
\begin{align}
L_h (v-\hat{v})&=0,\\
\sum_{x_{jkl}\in\gamma} [v_{jkl}-\hat{v}_{jkl}] \Phi_{jkl}(x^*,y^*,z^*) &=0.
\end{align}
With well-posedness assumed, we conclude that $v_{jkl}=\hat{v}_{jkl}$ in $N^+$, thus also in $\gamma$.
\end{proof}

Once $\hat{v}_\gamma$ is obtained by solving \eqref{eqn:inf_hom}, the approximate solution in $M^+$ is given by the discrete generalized Green's formula:
\begin{align}\label{eqn:discrete_green}
u_{jkl} = G_h[\chi_{M^-}L_h\hat{v}_\gamma]+G_h[\chi_{M^+}f_h],\quad x_{jkl}\in M^+
\end{align}
where we only replaced $v_\gamma$ with $\hat{v}_\gamma$ in the discrete generalized Green's formula \eqref{eqn:discrete_green_v}.

To summarize, the operator $\widehat{P}_\gamma$ (essentially $\widehat{P}_{\gamma_+}$) is constructed using explicit knowledge of the lattice Green's function, thus the excessive computational cost mentioned in \cite{xia2023local} for computing the difference potentials operator in 3D is avoided. Only two instances that involve the operator $G_h$ in the finite auxiliary domain will be based on 3D FFTs in \eqref{eqn:discrete_green}.

%% file: tex/bae.tex
\section{Boundary Algebraic Equations}\label{sec:bae}
In this section we focus on the reduced boundary equations with projections or the boundary algebraic equations \eqref{eqn:inf_veq}. We shall decompose the operator $\widehat{P}_{\gamma_+} $ in \eqref{eqn:beq} into two parts $P_+$ and $P_{-}$ that correspond to columns in $\gamma_+$ and $\gamma_-$:
\begin{align}
v_{\gamma_+} - \left[\begin{array}{cc}P_+ & P_-\end{array}\right]\left(\begin{array}{cc}v_{\gamma_+} & v_{\gamma_-}\end{array}\right)^T = 0
\end{align}
This gives a relation between $v_{\gamma_+}$ and $v_{\gamma_-}$:
\begin{align}
(I-P_+)v_{\gamma_+} = P_- v_{\gamma_-}\label{eqn:beq_bae}
\end{align}
where $P_+$ is a square matrix of size $|\gamma_+|\times|\gamma_+|$, $P_-$ is of size $|\gamma_+|\times|\gamma_-|$. Note that $P_+$ and $P_-$ are both dense. Even with the lattice Green's function, the computation of $P_+$ and $P_-$ and their storage can be expensive in 3D. Now we aim to establish some boundary algebraic equations like \eqref{eqn:beq_bae} directly from the lattice Green's function. The BAE \eqref{eqn:beq_bae} can also be regarded as the discrete version of Dirichlet to Neumann (DtN) map.

\subsection{Direct formulation}
Similarly to the boundary integral method, we will introduce the concept of direct and indirect formulations, where direct formulation is based on Green's third identity while indirect formulation is not. In this classification, the boundary algebraic equation \eqref{eqn:beq_bae} is a direct formulation. To see this, we introduce a lattice version of Green's second identity from \cite[Sec. 5, Lemma 1]{duffin1953discrete}:
\begin{lemma}
Let $u$ and $v$ be lattice/grid functions and let $E$ be a finite set of lattice points, then
\begin{align}
\sum_E (vLu - uLv) = \sum_S[v(p)u(q)-u(p)v(q)]
\end{align}
where $\sum_E$ denotes the summation over $E$, and $\sum_S$ denotes the summation over the set $S$ of pairs of points $(p,q)$ that are neighbors, $p$ being in $E$ and $q$ being in the complement of $E$.
\end{lemma}

Now we choose a lattice harmonic function $u$ such that $Lu = 0$ and $v(m)=G(m-p)$ the lattice Green's function of the difference operator $L$ at peaked at $p$. Let $p$ be any point in $\gamma_+$ and $q$ be any point in $\gamma_-$ and focus on set $p$, then the lattice version of Green's third identity would be 
\begin{align}\label{eqn:direct_bae}
% -u(p^*) = \sum_S[G(p-p^*)u(q)-u(p)G(q-p^*)]\\
v_{\gamma_+} - G_{\gamma_-}v_{\gamma_+} = -G_{\gamma_+}v_{\gamma_-}
\end{align}
which is exactly the form of BAE \eqref{eqn:beq_bae}, if one imposes $P_+:=G_{\gamma_-}$ and $P_-:=-G_{\gamma_+}$ that denote the pair interactions of $\gamma_+$ and $\gamma_-$. The matrix $G_{\gamma_-}$ needs proper reduction from $\gamma_-$ to $\gamma_+$. For example, if $q_1$ and $q_2$ both are neighbors of $p$, then the coefficient of $u(p)$ will be the addition of $G(q_1-p)$ and $G(q_2-p)$. Similarly, the matrix $G_{\gamma_+}$ will need to be extended from $\gamma_+$ to $\gamma_-$, as at both $q_1$ and $q_2$, $u(q)$ has the same coefficient.

\subsection{Indirect formulation}
As in the indirect boundary integral approach, we investigate indirect boundary algebraic method. We follow the work of boundary algebraic equations \cite{martinsson2009boundary}, but allow misalignment of grid points with the continuous boundary $\partial\Omega$.

We will start with the kernel for single layer potential:
\begin{align}\label{eqn:single}
S(m,n) = G_\sigma(|m-n|),\quad m,n\in \mathbb{Z}^3
\end{align} 
where $G_\sigma$ is the lattice Green's function in the free space.

Note that unlike the continuous case (the boundary integral method), no singularity appears when $m=n$ in \eqref{eqn:single}. The single layer potential for any density $q$ defined on $\gamma_-$ (source) is
\begin{align}\label{eqn:single_potential}
v_m=\sum_{x_n\in\gamma_-}S(m,n)q(n),\quad x_m\in N^+
\end{align}
where the target is any point in $N^+$.

\begin{proposition}\label{prop:single_potential}
The single layer potential \eqref{eqn:single_potential} satisfies the homogeneous difference equations $L_hv_m=0$ for $x_m\in M^+$.
\end{proposition}
\begin{proof}
For $x_m\in M^+$,
\begin{subequations}
\begin{align}
L_h v_m =& L_h\sum_{x_n\in\gamma_-}S(m,n)q(n)\\
=&\sum_{x_n\in\gamma_-}\Big[(-6+\sigma h^2)S(m,n)+S(m+e_1,n)+S(m-e_1,n)\nonumber\\
&+S(m+e_2,n)+S(m-e_2,n)+S(m+e_3,n)+S(m-e_3,n)\Big]q(n)\\
=&\sum_{x_n\in\gamma_n}\delta(m,n)q(n)\\
=&0
\end{align}
as $m$ and $n$ belong to disjoint sets. Here $e_i$, $(i=1,2,3)$ are unit vectors in the $x,y,z$ directions respectively.
\end{subequations}
\end{proof}

Now assume for the source defined on $\gamma_-$ with single layer density $q_{s}$, we look at target sets $\gamma_+$ and $\gamma_-$ with matrix-vector notation respectively:
\begin{align}
v_{\gamma_+} =S_{+}q_{s},\quad v_{\gamma_-} =S_{-}q_{s}
\end{align}
where $S_+$ is of size $|\gamma_+|\times|\gamma_-|$ and $S_-$ is of size $|\gamma_-|\times|\gamma_-|$.
Then $v_{\gamma_-}$ and $v_{\gamma_+}$ are related through the unknown density $q_{s}$, i.e.,
\begin{align}
v_{\gamma_+} = S_{+}S^{-1}_{-}v_{\gamma_-}.
\end{align}

\begin{remark}
The set  $\gamma_-$ can be regarded as the grid boundary of the set $M^+$ and spectral properties of $S_-$ in \cite{martinsson2009boundary} can be applied. In addition, $S_-$ is invertible, symmetric and diagonally dominant, which is useful for theoretical studies.
\end{remark}

Now for the double layer kernel, we adopt the same definition in \cite{martinsson2009boundary}:
\begin{align}
    D(m,n):=\sum_{k\in \mathbb{D}_n} G_\sigma(|m-n|) - G_\sigma(|m-k|)
\end{align}
where $G_\sigma$ is the free space lattice Green's function and $\mathbb{D}_n$ is the set of nodes outside of $N^+$ but connected to source node $n$. The set $\gamma_-$ is regarded as the grid boundary for $M^+$. 

The double layer potentials for a grid density $q_d$ defined on $\gamma_-$ is 
\begin{align}\label{eqn:double_potential}
v_m=\sum_{x_n\in\gamma_-}D(m,n)q_d(n),\quad x_m\in N^+,
\end{align}
where $q_d$ denotes the double layer density.

\begin{proposition}
The double layer potential \eqref{eqn:double_potential} satisfies the homogeneous difference equations $L_hv_m=0$ for $x_m\in M^+$.
\end{proposition}
We will omit the proof here as it is similar to the proof of Proposition \ref{prop:single_potential}.

With restriction to $\gamma_+$ and $\gamma_-$ respectively, we have
\begin{align}
v_{\gamma_+}=D_+q_{d},\quad v_{\gamma_-}=D_-q_{d}
\end{align}
If the square matrix $D_-$ is invertible, we also have
\begin{align}
v_{\gamma_+} = D_+D_{-}^{-1}v_{\gamma_-}.
\end{align}
from similar arguments in the single layer case.

\begin{proposition}
The following relations hold:
\begin{align}\label{eqn:equal}
(I-P_+)^{-1}P_- = S_+S_{-}^{-1}=D_+D_{-}^{-1}
\end{align}
\end{proposition}

The equivalence~\eqref{eqn:equal} essentially connects point set $\gamma_{+}$ with point set $\gamma_{-}$, which also agrees with the difference equation:
\begin{align}
L_hu_h = 0, \quad (x_i,y_j,z_k)\in M^+.
\end{align}
Once the lattice boundary values at $\gamma_{-}$ are set, the above difference equation admits a unique solution, hence values at $\gamma_{+}$ are uniquely determined.

\begin{remark}
The involved matrices $P_\pm$, $S_\pm$ and $D_\pm$ are readily obtainable from the lattice Green's function without explicitly solving the auxiliary problem, which is desirable when designing matrix-free algorithms.
\end{remark}

\begin{remark}
The infinite lattice only comes into play conceptually when we compute the values of lattice Green's functions.
\end{remark}

\subsection{Spectral property}

In complete analogous to the 2D results in \cite{martinsson2009boundary}, we state the spectral property of the matrices $S_-$ and $D_-$.

\begin{theorem}
Let $\Omega_h$ be a connected domain in $\mathbb{Z}$ whose boundary $\gamma_-$ contains $N$ boundary nodes. Let $S_-$ be the $N\times N$ ($N=|\gamma_-|$) matrix corresponding to the linear system equivalent to the single-layer potential in 3D. Then, any singular value $\lambda$ of $S_-$ satisfies:
\begin{align}
\frac{1}{12} \leq  \lambda \leq CN
\end{align}
where $C$ is a dimension-independent constant.
\end{theorem}

\begin{proof}
For the upper bound:

\begin{align}
\lambda  \leq ||S_-||_2 \leq \sqrt{||S_-||_1||S_-||_{\infty}}
\end{align}
Since $|S(m,n)| \leq C$, we have $||S_-||_1 \leq CN$ and $||S_-||_\infty \leq CN$
Therefore:
\begin{align}
\lambda\leq ||S_-||_2 \leq \sqrt{CN\cdot CN} = CN
\end{align}

For the lower bound, we use an energy argument similar to the 2D case \cite{martinsson2009boundary} but extended to three dimensions.
Define the forward and backward difference operators:
\begin{subequations}
\begin{align}
\partial_i u = u(m + e_i) - u(m)\\
\bar{\partial}_i u = u(m) - u(m - e_i)
\end{align}
\end{subequations}
where $e_i$ are unit vectors in the three coordinate directions.
For lattice potentials $u, v$ that decay sufficiently fast at infinity:
\begin{align}
\sum [\bar{\partial}_iu](m)v(m) = -\sum u(m)[\partial_iv](m)
\end{align}
For any boundary density $\sigma$, we need to prove:
\begin{align}
\left|\sum_{m\in\gamma_-} [S_-\sigma](m) \sigma(m)\right|\geq \frac{1}{12}\sum_{m\in\gamma_-} |\sigma(m)|^2
\end{align}

Let $u$ be defined by
\begin{align}
u(m) = \sum_{n\in\gamma_-} S(m,n)\sigma(n)
\end{align}
Then
\begin{subequations}
\begin{align}
\sum_{m\in\gamma_-} [S_-\sigma](m) \sigma(m) = \sum_{m\in\gamma_-} u(m)\sigma(m) = \sum_{m\in\gamma_-} u(m)Au = \sum_{m\in\mathbb{Z}^3} u(m)Au\\
=\sum_{m\in\mathbb{Z}^3} |\partial_1u|^2+|\partial_2u|^2+|\partial_3u|^2\\
=\frac{1}{2}\sum_{m\in\mathbb{Z}^3} |\partial_1u|^2+|\partial_2u|^2+|\partial_3u|^2+|\bar{\partial}_1u|^2+|\bar{\partial}_2u|^2+|\bar{\partial}_3u|^2\\
\geq\frac{1}{2}\sum_{m\in\gamma_-} |\partial_1u|^2+|\partial_2u|^2+|\partial_3u|^2+|\bar{\partial}_1u|^2+|\bar{\partial}_2u|^2+|\bar{\partial}_3u|^2
\end{align}
\end{subequations}
The flux equilibrium equation at boundary points becomes:
\begin{align}
-\partial_1 u(m)+\bar{\partial_1}u(m)-\partial_2 u(m)+\bar{\partial_2}u(m)-\partial_3 u(m)+\bar{\partial_3}u(m)=\sigma(m)
\end{align}
which implies
\begin{align}
|\partial_1u|^2+|\partial_2u|^2+|\partial_3u|^2+|\bar{\partial}_1u|^2+|\bar{\partial}_2u|^2+|\bar{\partial}_3u|^2\geq |\sigma(m)|^2
\end{align}
Combining these inequalities gives
\begin{align}
\left|\sum_{m\in\gamma_-} [S_-\sigma](m) \sigma(m)\right|\geq \frac{1}{12}\sum_{m\in\gamma_-} |\sigma(m)|^2
\end{align}

\end{proof}

For double layer potential, similar results can be established as well.
\begin{theorem}
Let $\gamma_-$ be the boundary of a connected lattice domain $\Omega_h$ in $\mathbb{Z}^3$ with $N$ boundary nodes, and $c$ be a real number such that $0<c\leq 1$. Suppose there exists an ordering $\{m_i,i=1,2,\cdots,N\}$ of $\gamma_-$ such that for all $i\neq j$,
\begin{align}
|m_i-m_j|\geq c\cdot d(i,j)
\end{align}
where $d(i,j)=\min(|i-j|,(N-|i-j|))$, the the matrix $D_-$ of the double layer potential associated with $\gamma_-$ satisfies
\begin{align}
||D_-||\leq \frac{C}{c^2}
\end{align}
where $C$ is a dimension-independent constant.
\end{theorem}

\begin{proof}
Likewise, one can verify that the double layer kernel $D$ satisfies
\begin{align}
|D(m,n)|\leq \frac{C}{1+|m-n|^2}
\end{align}
for all $m,n\in\gamma_-$.

Combining with the geometric assumption
\begin{align}
|D_{-,ij}| = |D(m_i,m_j)|\leq \frac{C}{1+|m_i-m_j|^2}\leq \frac{C}{1+c^2d^2(i,j)}\leq \frac{1}{c^2}\frac{C}{1+d^2(i,j)}
\end{align}
Now define $D'$ with entry at $(i,j)$
\begin{align}
D'_{ij}=\frac{1}{1+d^2(i,j)}
\end{align}
Hence $D'$ is a symmetric circulant matrix, whose eigenvalues can be shown to be uniformly bounded by a constant $C$ independent of $N$ using Fourier transform.
\begin{align}
||D'||\leq C
\end{align}
thus
\begin{align}
||D_-||\leq \frac{1}{c^2}||D'||\leq \frac{C}{c^2} 
\end{align}
\end{proof}

\begin{remark}
It can be checked that the double layer potential matrix $D$ admits an eigenvalue $-1$.
\end{remark}

\subsection{Boundary equations}

Now we have obtained three types of boundary equations
\begin{subequations}
\begin{align}
&\left\{
\begin{array}{r}
(I - P_+)v_{\gamma_+} - P_{-}v_{\gamma_-}=0\\
\Phi_+v_{\gamma_+} + \Phi_{-}v_{\gamma_-}=b
\end{array}
\right.\label{eqn:direct_sys}\\
&\left\{
\begin{array}{r}
v_{\gamma_+} =S_{+}q_{s}\\
v_{\gamma_-} =S_{-}q_{s}\\
\Phi_+v_{\gamma_+} + \Phi_{-}v_{\gamma_-}=b
\end{array}
\right.\label{eqn:single_sys}\\
&\left\{
\begin{array}{r}
v_{\gamma_+} =D_{+}q_{d}\\
v_{\gamma_-} =D_{-}q_{d}\\
\Phi_+v_{\gamma_+} + \Phi_{-}v_{\gamma_-}=b
\end{array}\label{eqn:double_sys}
\right.
\end{align}
\end{subequations}
where $b(x^*,y^*,z^*)=g(x^*,y^*,z^*)-\sum_{x_{jkl}\in\gamma}[G_hf_{\gamma}]_{jkl} \Phi_{jkl}(x^*,y^*,z^*)$ for $(x^*,y^*,z^*)\in\Gamma$, and $\Phi_\pm$ are the sub-matrices in the coefficient matrices $\Phi$ that corresponds to point set $\gamma_\pm$ respectively.

\paragraph{Schur complement}

The three systems \eqref{eqn:direct_sys}--\eqref{eqn:double_sys} hint a block linear system
\begin{align}
\left(
\begin{array}{cc}
A & B \\
\Phi_+ & \Phi_- \\
\end{array}
\right)
\left(
\begin{array}{c}
v_{\gamma_+}\\
v_{\gamma_-}
\end{array}
\right)
=
\left(
\begin{array}{c}
0 \\
b \\
\end{array}
\right)
\end{align}
where $\Phi_+, \Phi_-$ are sparse and almost diagonal matrices.

Schur complement for the above system gives
\begin{subequations}
\begin{align}
v_{\gamma_+} &= -A^{-1}Bv_{\gamma_-}\\
v_{\gamma_-} &= (\Phi_--\Phi_+A^{-1}B)^{-1}b
\end{align}
\end{subequations}

The matrix multiplication $A^{-1}B$ can be any term in the equivalence \eqref{eqn:equal}:
\begin{align}
-A^{-1}B:=(I-P_+)^{-1}P_- = S_+S_{-}^{-1}=D_+D_{-}^{-1}
\end{align}

We have seen that the norm of $S_-$ will grow like $\mathcal{O}(N)$ and the norm of $D_-$ remain uniformly bounded. Previous numerical study in \cite{xia2023local} revealed that the condition number of $I-P_+$ will also grow linearly. Hence, the double layer formulation using $D_\pm$ is preferred.

When double layer formulation is used, the block system becomes:
\begin{align}
\left(
\begin{array}{cc}
D_+ & -D_- \\
\Phi_+ & \Phi_- \\
\end{array}
\right)
\left(
\begin{array}{c}
v_{\gamma_+}\\
v_{\gamma_-}
\end{array}
\right)
=
\left(
\begin{array}{c}
0 \\
b \\
\end{array}
\right)
\end{align}
and Schur complement gives
\begin{subequations}
\begin{align}
v_{\gamma_+} &= D_+D_-v_{\gamma_-}\\
v_{\gamma_-} &= (\Phi_-+\Phi_+D_+D_-^{-1})^{-1}b
\end{align}
\end{subequations}

The matrix $\Phi_--\Phi_+A^{-1}B$ or $\Phi_-+\Phi_+D_+D_-^{-1}$ resembles the Fredholm integral of second kind. This formulation has favorable conditioning properties. Notably, its condition number remains relatively stable under mesh refinement. Similar observations were made in \cite{feng2020fft,ren2022fft,ying2007kernel,gillis2018fast,gillis20192d,li1998fast,tan2009fast} where either Schur complement or Sherman-Morrison-Woodbury formula is used. Hence, $D_{-}$ and $\Phi_-+\Phi_+D_+D_-^{-1}$ are both well-conditioned, with condition numbers independent of mesh refinement, which are both suitable for iterative solvers such as GMRES.

\paragraph{Linear solve}
However, $I-P_{\gamma_+}$ is relatively small in size yet dense, so are $S_{-}$ and $D_{-}$, and it would be unwise to invert those matrices and construct $\Phi_-+\Phi_+D_+D_-^{-1}$. Instead, in the indirect formulation, we observe that
\begin{align}
v_{\gamma_+}=D_+q_d,\quad v_{\gamma_-}=D_-q_d\quad\mbox{or}\quad v_{\gamma_+}=S_+q_s,\quad v_{\gamma_-}=S_-q_s
\end{align} 
where $q_d$ and $q_s$ are the unknown double layer density and single layer density to be determined.

Then the single layer formulation together with the boundary condition gives
\begin{align}
(\Phi_+S_++\Phi_-S_-)q_s = b
\end{align}
and double layer gives
\begin{align}\label{eqn:double_bae}
(\Phi_+D_++\Phi_-D_-)q_d = b
\end{align}
where we can solve for the density $q_s$ or $q_d$.

Both systems can be solved using iterative methods. This approach avoids explicit storage or inversion of matrices $D_-$ and $S_-$. After obtaining the density $q_s$ or $q_d$, we can compute the density $v_\gamma$ through either the single or double layer formulation. Once $v_\gamma$ is obtained, the rest follows what we discussed in the difference potentials method in Section~\ref{sec:dpm}. Even though both single layer formulation and double layer formulation can be adopted, we will use the double layer formulation as it gives better conditioned linear system, where iterative solvers can be efficiently utilized.

\begin{remark}
The Schur complement admits better spectral property and this can be explained by the right preconditioning of BAE \eqref{eqn:double_bae}:
\begin{align}
(\Phi_+D_++\Phi_-D_-)(D_-^{-1}D_-)q_d = b
\end{align}
which gives
\begin{align}
(\Phi_+D_+D_-^{-1}+\Phi_-)v_{\gamma_-} = b
\end{align}
However, $D_-$ is dense and using $D_-$ as a preconditioner might outweigh the advantages of using iterative solvers.
\end{remark}

\subsection{Outline of the developed unfitted BAE}

Now we summarize the developed unfitted BAE.

\begin{algorithm}[t]
\caption{Unfitted Boundary Algebraic Method}\label{alg:bae}
\allowdisplaybreaks
\begin{algorithmic}[1]

% \Procedure{UnfittedBAE}{$\Omega$, $\Omega^0$, $f$, $g$, $h$, $\sigma$}
    \State \textbf{Input:} Domain $\Omega$, Auxiliary domain $\Omega^0$, RHS $f$, BC data $g$, grid size $h$, coefficient $\sigma$
    \State \textbf{Output:} Solution $u$ on $M^+$

    \Statex

    \State // Precompute Lattice Green's Function on infinite lattices
    \State $\text{LGF} \gets \text{PrecomputeLGF}(\sigma)$ 

    \Statex

    \State // Initialize finite grid and point sets
    \State $\text{finite grid } \omega_h \gets \text{CreateCartesianGrid}(\Omega^0, h)$
    \State $M^\pm, N^\pm, \gamma^\pm \gets \text{ClassifyPoints}(\text{grid}, \Omega^0, \Omega)$
    
    \Statex

    \State // Compute particular solution on finite grid $\omega_h$
    \State $u_p \gets G_h[\chi_{M^+}f_h]$ in $N^+$ via FFT

    \Statex
    
    \State // Form boundary system (double layer formulation)
    \State $\Phi_\pm \gets \text{ConstructBasisMatrices}(\gamma^+, \gamma^-)$ \Comment{Sparse matrices}
    \State $D_\pm \gets \text{ConstructDoubleLayerOperators}(\text{LGF}, \gamma^+, \gamma^-)$ \Comment{Matrix-free}
    
    \Statex

    \State // Solve boundary system
    \State $b \gets g - \text{EvaluateBasisFunctions}(u_p)$ \Comment{Subtract contribution of $u_p$ at boundary points}
    \State $A \gets \Phi_+D_+ + \Phi_-D_-$ \Comment{Matrix-free}
    \State $q_d \gets \text{GMRES}(A, b)$  \Comment{Double layer density}
    \State $u_{\gamma} \gets [D_+q_d; D_-q_d]$ \Comment{Matrix-free}

    \Statex
    
    \State // Reconstruct solution
    \State $u \gets u_p + G_h[\chi_{M^-}L_hu_\gamma]$ in $M^+$

\end{algorithmic}
\end{algorithm}

The memory requirement of Algorithm~\ref{alg:bae} will be $\mathcal{O}(N^3)$. The matrix-free multiplication of $D_+q_d$ and $D_-q_d$ would benefit from fast summation techniques, which is not explored in this work. The computational cost in Algorithm~\ref{alg:bae} mainly lies in the iterative solvers for finding the double layer density and the 3D FFT of finding particular solution and the construction of difference potentials once $u_\gamma$ is obtained. In this work, we use FFT to solve the for particular solution, utilizing the bounded domain. For unbounded domains and nonhomogenous source functions, fast summation techniques such as those developed in \cite{gillman2014fast} is in order.

\begin{remark}
The double layer formulation $D_{\pm}$ in Algorithm~\ref{alg:bae} can be replaced by the single layer formulation $S_\pm$ or the direct formulation in the difference potentials framework, depending on the computational need.
\end{remark}

%% file: tex/numerics.tex
\section{Numerical results}\label{sec:numerical_results}
We will use Algorithm~\ref{alg:bae} to perform the numerical tests.

The finite auxiliary domain is slightly larger than the torus and is chosen to be $[-1-\ell,1+\ell]^3$ and the grid size is taken as
\begin{align}
h = \frac{2+2\ell}{N},\quad N = 2^n-1
\end{align}
which gives a total of $N^3$ degrees of freedom in the auxiliary domain and the number of unknowns in the boundary equations is in the scale of $\mathcal{O}(N^2)$.

The Lattice Green's functions are precomputed on mesh $\mathbb{Z}^3$ and stored in a lookup matrix of size $N\times N\times N$ for only positive index values, and can be reused in different tests. Due to the symmetry of the Lattice Green's function, only 1/48 of the values are stored in 3D.

\subsection{Poisson's equation}
The geometry is an ellipsoid given implicitly by
\begin{align}
\phi(x,y,z) = \frac{x^2}{a^2}+\frac{y^2}{b^2}+\frac{z^2}{c^2}-1=0
\end{align}
with $a=1,b=0.8,c=0.4$. We choose $\ell=0.25$ for the auxiliary domain in this test. The spherical harmonics spectral approach developed in \cite{epshteyn2019efficient} will not work well for this ellipsoidal geometry.

The errors are computed against the exact solution
\begin{align}
u(x,y,z) = x^2+y^2+z^2
\end{align}
and boundary data $g(x,y,z)$ and the source term $f(x,y,z)$ are computed using the exact $u(x,y,z)$.

Due to memory limit, we do not store matrices $D_{\pm}$ or $S_\pm$, and we refer to \cite{martinsson2009boundary} for behaviors of singular values of matrices $S_-$ and $D_-$. In Figure~\ref{fig:poisson_res}, we present the relative residual of solving 
\begin{align}
(\Phi_+D_+ + \Phi_-D_-) q_d = b
\end{align}
using vanilla GMRES with no preconditioner and a zero initial guess. The tolerance is set to be $10^{-14}$ for test purpose, whereas a larger tolerance such as $10^{-8}$ would suffice for accuracy. It can be observed that for $N=31,63,127,255$, where the length of $q_d$ approximately quadruples over mesh refinement, the number of iterations grow mildly, even though not as good as the double layer formulation in the boundary integral method. This indicates that preconditioning techniques should be explored.

\begin{figure}[H]
    \centering
    % trim=left bottom right top
    \includegraphics[width=0.5\textwidth, trim = 0cm 5cm 0cm 5.5cm]{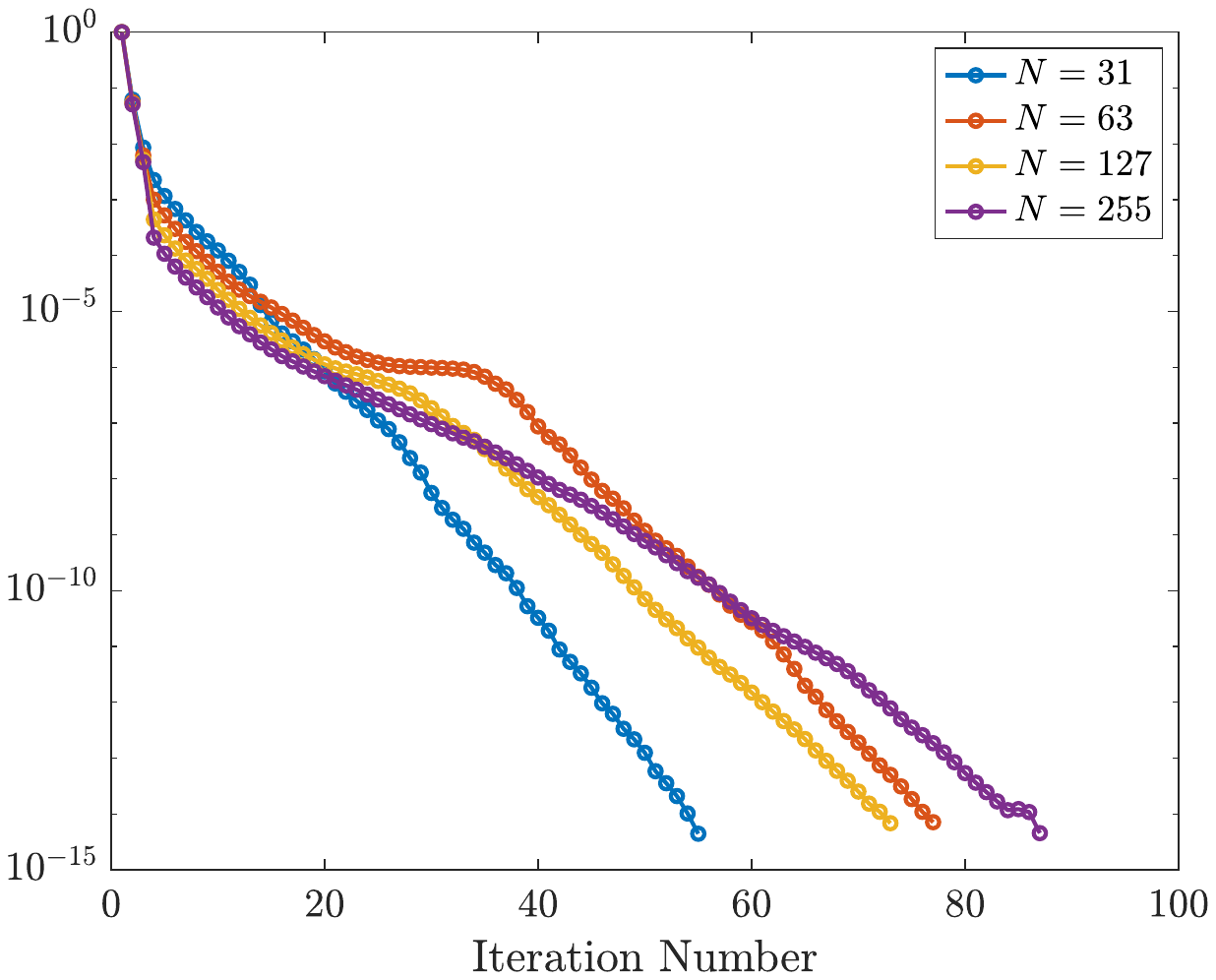}
    \caption{Relative residual of GMRES using double layer formulation for Poisson equations}
    \label{fig:poisson_res}
\end{figure}

In Figure~\ref{fig:poisson_results}, we plot the numerical solution and the error patterns on the finest mesh $255\times255\times255$ in our simulation. The error patterns hint that it might benefit from post-processing techniques studied such as in \cite{mirzaee2011smoothness}.

\begin{figure}[H]
    \centering
    \begin{subfigure}{0.45\textwidth}
        \centering
        % trim=left bottom right top
        \includegraphics[width=\textwidth]{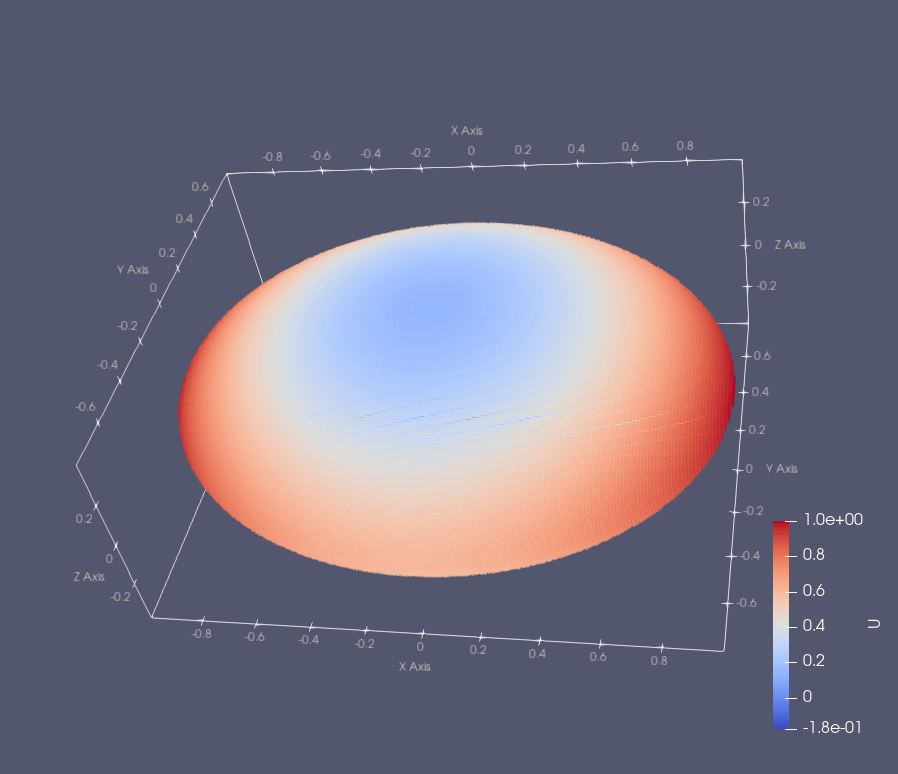}
        % \vspace*{0.5mm}
        \caption{Numerical Solution $(N=255)$}
        \label{fig:poisson_solution}
    \end{subfigure}
    \quad
    \begin{subfigure}{0.45\textwidth}
        \centering
        \includegraphics[width=\textwidth]{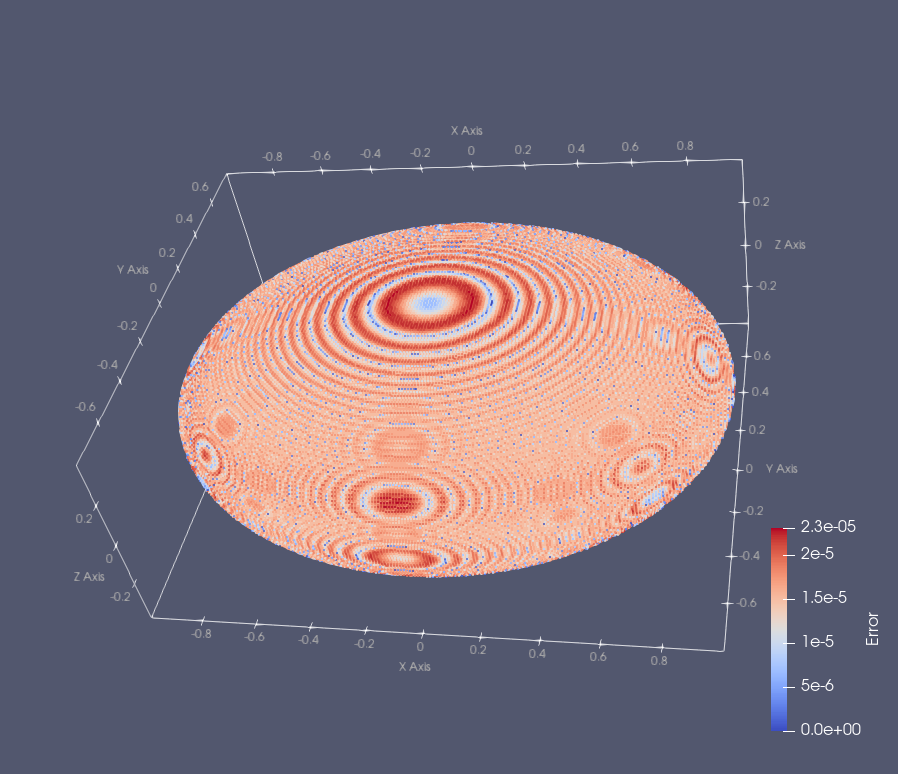}
        % \vspace*{0.5mm}
        \caption{Pointwise errors $(N=255)$}
        \label{fig:poisson_error}
    \end{subfigure}
    \caption{Numerical solution and pointwise errors for Poisson equation}\label{fig:poisson_results}
\end{figure}

In Table~\ref{table:poisson_convergence}, the max-norm error and convergence rates are presented, which shows the designed second order convergence.

\begin{table}[htbp]
    \centering
    \begin{tabular}{@{} r S[table-format=1.4e-1] S[table-format=1.2] @{}}
        \toprule
        {$N$} & {Max Error} & {Rate} \\
        \midrule
        31   & 1.3222e-03 & {--} \\
        63   & 3.5101e-04 & 1.91 \\
        127  & 9.1688e-05 & 1.94 \\
        255  & 2.2926e-05 & 2.00 \\
        \bottomrule
    \end{tabular}
    \caption{Max error and convergence rates for Poisson's equation in an ellipsoid}
    \label{table:poisson_convergence}
\end{table}

\subsection{Modified Helmholtz equation}

For modified Helmholtz equation, we test with $\sigma=10$. For the geometry, we use a multi-connected torus defined by the implicit function
\begin{align}
\phi(x,y,z) = (\sqrt{x^2+y^2}-R)^2+z^2-r^2 = 0
\end{align}
where $R=0.6$ is the distance between the center of the tube and the center of the torus and $r=0.3$ is the radius of the tube. The interior of the torus is categorized by $\phi(x,y,z)<0$. The auxiliary domain takes $\ell=0.1$.

\begin{remark}
A torus is simple enough, yet presenting sufficient numerical challenges for unfitted boundary methods. We chose this geometry to demonstrate the effectiveness of combing lattice Green's function in free space with local basis functions, where global basis functions will be difficult to construct for spectral approaches. We should also mention that using Non-Uniform Rational B-Splines (NURBS) on patches for 3D CAD geometry is also possible and has been studied in the difference potentials framework in \cite{PETROPAVLOVSKY2024112705}.
\end{remark}

The exact solution is
\begin{align}
u(x,y,z) = \sin(x)\cos(y)\sin(z),
\end{align}
and $f$ and $g$ are computed using the exact solution.

It can be seen in Figure~\ref{fig:poisson_res} and Figure~\ref{fig:mod_res} that the iteration numbers are similar regardless of the equation types $\sigma=0$ or $\sigma=10$ or different types of geometry or lattice Green's functions, corroborating the robustness of the boundary algebraic linear system.

The GMRES convergence behavior shown in Figure~\ref{fig:mod_res} demonstrates the effectiveness of the double layer formulation for the modified Helmholtz equation. The similar convergence patterns observed for different mesh sizes $(N = 31, 63, 127, 255)$ suggest that the conditioning of the system remains relatively stable under mesh refinement. This behavior is particularly noteworthy given the complex geometry of the torus and the challenging nature of the modified Helmholtz operator.

\begin{figure}[H]
    \centering
    % trim=left bottom right top
    \includegraphics[width=0.5\textwidth, trim = 0cm 5cm 0cm 5.5cm]{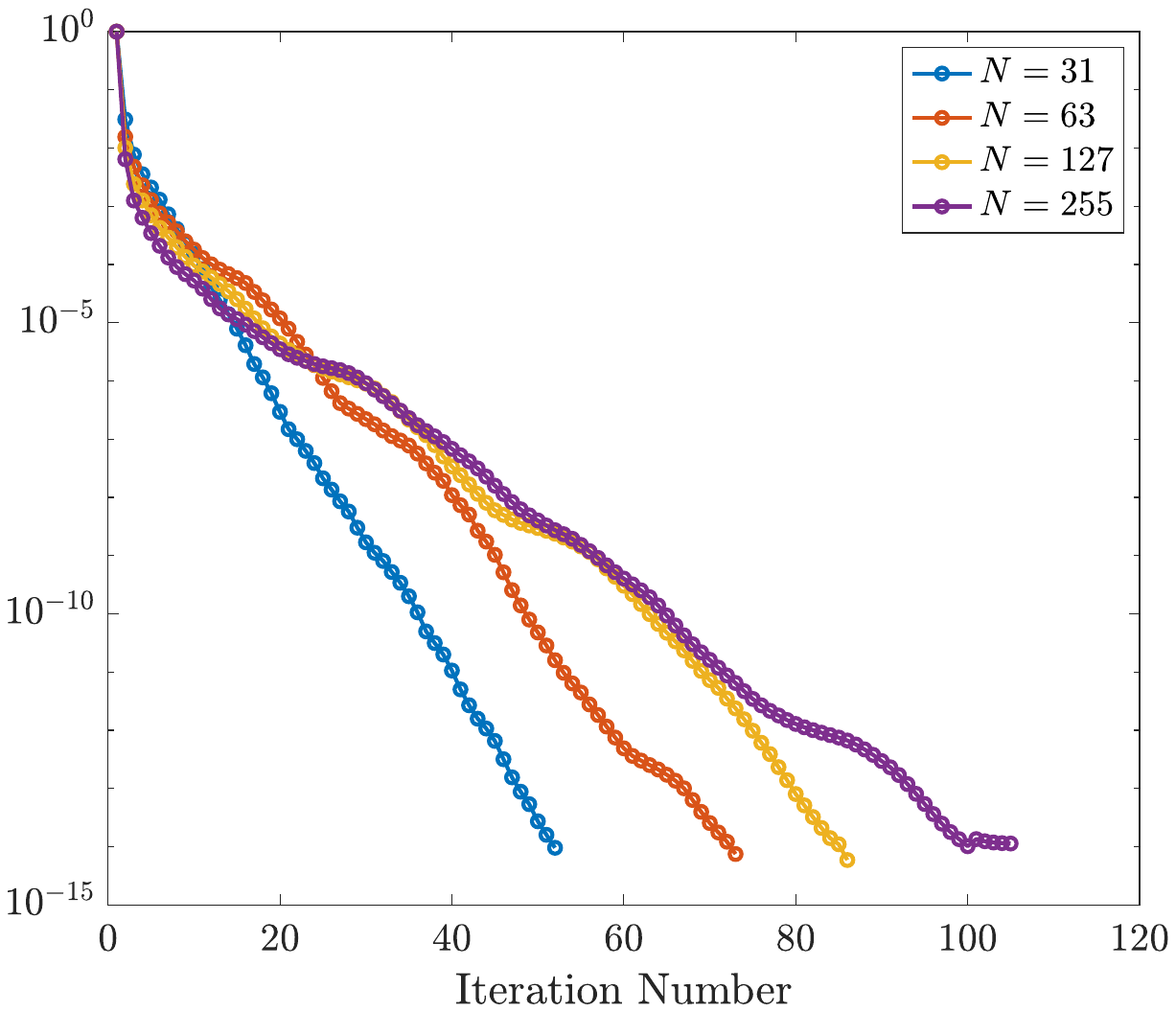}
    \caption{Relative residual of GMRES using double layer formulation for modified Helmholtz equations}
    \label{fig:mod_res}
\end{figure}

In Figure~\ref{fig:mod_results}, the numerical solution and the error patterns on the finest mesh $255\times255\times255$ are presented. The large errors occur around where $u$ is large in magnitude.

\begin{figure}[H]
    \centering
    \begin{subfigure}{0.45\textwidth}
        \centering
        % trim=left bottom right top
        \includegraphics[width=\textwidth]{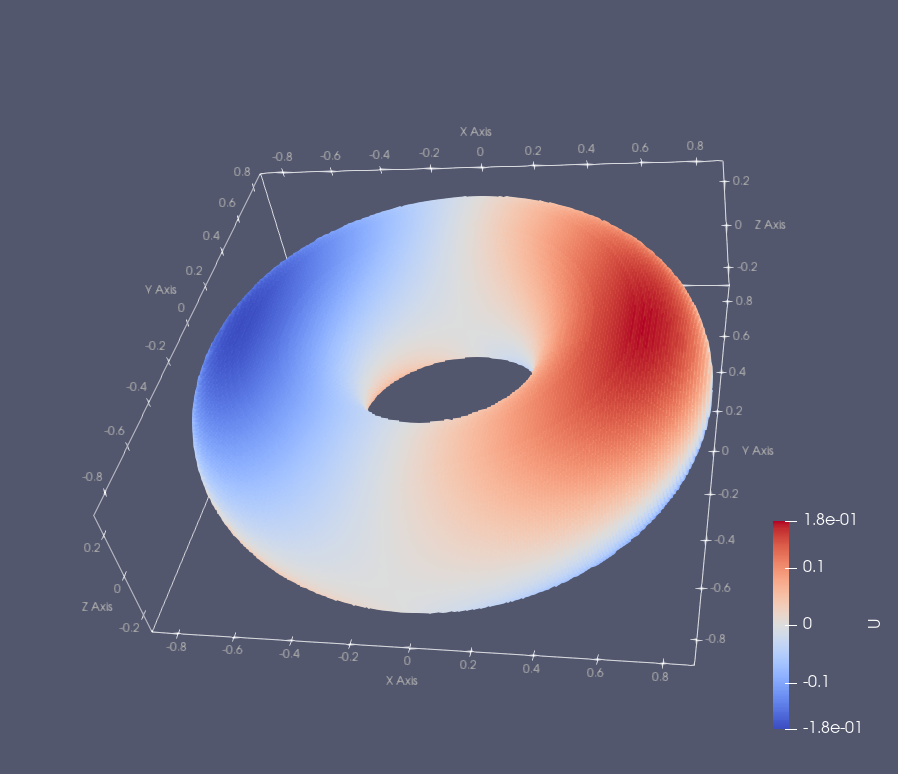}
        % \vspace*{0.5mm}
        \caption{Numerical Solution $(N=255)$}
        \label{fig:mod_solution}
    \end{subfigure}
    \quad
    \begin{subfigure}{0.45\textwidth}
        \centering
        \includegraphics[width=\textwidth]{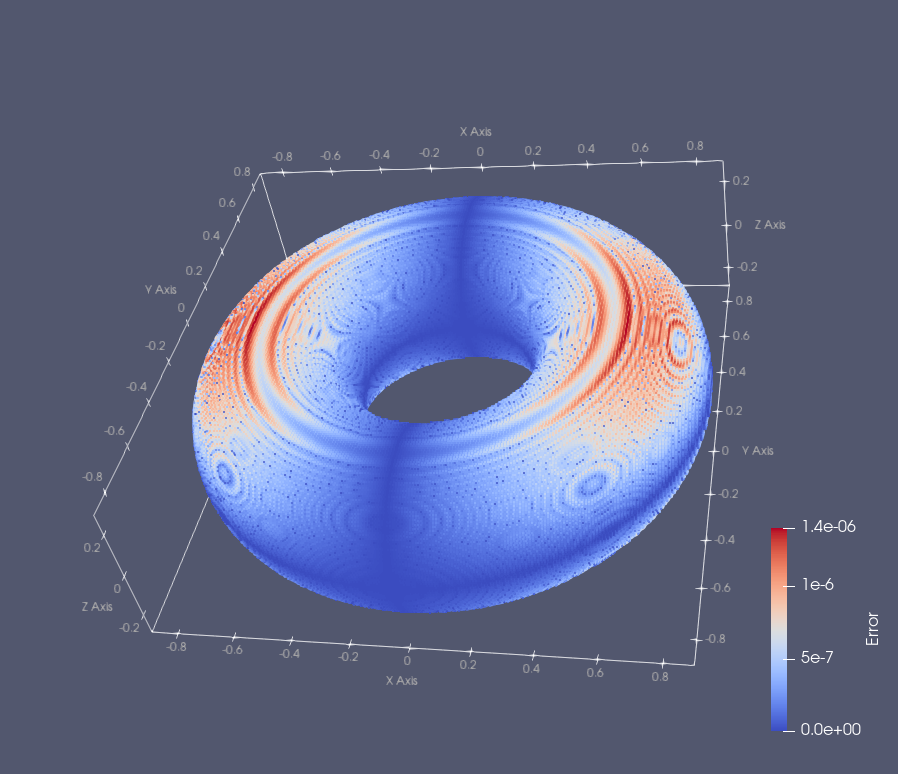}
        % \vspace*{0.5mm}
        \caption{Pointwise errors $(N=255)$}
        \label{fig:mod_error}
    \end{subfigure}
    \caption{Numerical solution and point errors for modified Helmholtz equation}\label{fig:mod_results}
\end{figure}

In Table~\ref{table:mod_convergence}, the max-norm error and convergence rates also show that the designed second order convergence is achieved.

\begin{table}[htbp]
    \centering
    \begin{tabular}{@{} r S[table-format=1.4e-1] S[table-format=1.2] @{}}
        \toprule
        {$N$} & {Max Error} & {Rate} \\
        \midrule
        31   & 7.5002e-05 & {--} \\
        63   & 2.0682e-05 & 1.86 \\
        127  & 5.6077e-06 & 1.88 \\
        255  & 1.3978e-06 & 2.00 \\
        \bottomrule
    \end{tabular}
    \caption{Max Error and Convergence Rates for modified Helmholtz equation in a torus}
    \label{table:mod_convergence}
\end{table}

%% file: tex/conclusion.tex
\section{Conclusion}\label{sec:conclusion}

We have developed an unfitted boundary algebraic equation method based on difference potentials for elliptic PDEs using lattice Green's functions that significantly advances the state of the art in handling complex 3D geometries. By replacing finite auxiliary domains with free-space LGFs, our approach dramatically reduces the computational complexity of constructing difference potentials operators, making three-dimensional simulations more tractable. The analytical derivation of LGF-based potentials, combined with rigorous proofs of equivalence between direct and indirect BAE formulations, provides a solid theoretical foundation for the method. Our spectral analysis demonstrates that the resulting systems are well-conditioned for iterative solvers, particularly when using double layer formulations. Numerical experiments confirm the method's efficiency and accuracy for both Poisson and modified Helmholtz equations in 3D implicitly defined geometry.

The developed approach successfully bridges the gap between structured grid efficiency and geometric flexibility, offering a robust foundation for multidisciplinary applications in computational physics and engineering. By combining the computational advantages of lattice Green's functions with the geometric flexibility of difference potentials, this method provides a promising path forward for solving complex PDEs in irregular domains with optimal efficiency. It should also be noted that the current approach only applies to constant-coefficient PDEs, while the difference potentials framework is suitable for variable-coefficient or nonlinear PDEs.

Future work will extend this framework to handle unbounded domains through fast multipole acceleration or FFT based libraries \cite{caprace2021flups} and explore applications to more challenging problems such as high-frequency Helmholtz equations using lattice Green's function for Helmholtz equation \cite{beylkin2009fast,beylkin2008fast,linton2010lattice} and Stokes flows. Additional developments will focus on implementing high-order stencils as studied in \cite{gabbard2024lattice}, developing fast summation techniques for matrix-free implementations similarly as in \cite{gillman2014fast}, and investigating various types of boundary conditions. 